\documentclass[11pt]{article}
\usepackage{amsmath, amssymb, amsthm}
\usepackage{enumerate}
\usepackage{fullpage}

\date{}
\title{\vspace{-0.8cm}Self-similarity of graphs}
\author{
Choongbum Lee \thanks{Department of Mathematics, UCLA, Los
Angeles, CA, 90095. Email: choongbum.lee@gmail.com. Research
supported in part by a Samsung Scholarship.}
\and
Po-Shen Loh \thanks{Department of Mathematical Sciences,
Carnegie Mellon University, Pittsburgh, PA 15213, e-mail: ploh@cmu.edu.
Research supported in part by an NSA Young Investigators Grant and a USA-Israel BSF grant.
}
\and
Benny Sudakov \thanks{Department of Mathematics, UCLA, Los Angeles, CA 90095.
Email: bsudakov@math.ucla.edu. Research supported in
part by NSF grant DMS-1101185, NSF CAREER award DMS-0812005 and by USA-Israeli BSF grant. }
}

\theoremstyle{plain}
\newtheorem{THM}{Theorem}[section]
\newtheorem{PROP}[THM]{Proposition}
\newtheorem{PROBLEM}[THM]{Problem}
\newtheorem{LEMMA}[THM]{Lemma}
\newtheorem{DEF}[THM]{Definition}
\newtheorem{COR}[THM]{Corollary}

\theoremstyle{definition}

\newcommand{\pr}[1]{\mathbb{P}\left[ #1 \right]}

\newcommand{\ISO}{\iota}
\newcommand{\Gnp}{G_{n,p}}
\newcommand{\bin}{\text{Bin}}

\begin{document}
\maketitle

%\begin{enumerate}[(i)]
%  \setlength{\itemsep}{1pt} \setlength{\parskip}{0pt}
%  \setlength{\parsep}{0pt}
%  \item
%\end{enumerate}

\begin{abstract}
  An old problem raised independently by Jacobson and Sch\"onheim asks
  to determine the maximum $s$ for which every graph with $m$ edges
  contains a pair of edge-disjoint isomorphic subgraphs with $s$
  edges. In this paper we determine this maximum up to a constant
  factor. We show that every $m$-edge graph contains a pair of
  edge-disjoint isomorphic subgraphs with at least $c (m\log m)^{2/3}$
  edges for some absolute constant $c$, and find graphs where this
  estimate is off only by a multiplicative constant.  Our results
  improve bounds of Erd\H{o}s, Pach, and Pyber from 1987.
\end{abstract}

\section{Introduction}

The decomposition of a given graph into smaller subgraphs is an old
problem in graph theory that has been studied from numerous
perspectives. A celebrated result of Wilson \cite{Wilson} asserts that
given any fixed graph $H$, the edge set of any sufficiently large
complete graph $K_n$ can be partitioned into edge-disjoint copies of
$H$, as long as the obvious necessary divisibility conditions $e(H)
\mid \binom{n}{2}$ and $g \mid n-1$ (where $g$ is the greatest common
divisor of the degrees of $H$) are satisfied.

%If one no longer specifies which graphs we wish to decompose a given
%graph into, then there is a

A \emph{factor}\/ of a graph is a spanning subgraph, and a
\emph{factorization}\/ is a partition of its edges into factors. A
series of papers by Graham, Harary, Robinson, Wallis, and Wormald
(see, e.g., \cite{iso-9, iso-8, iso-1, iso-2, iso-6}) introduced the
systematic study of \emph{isomorphic factorizations}, in which the
resulting factors are required to be isomorphic to each other as
graphs. In this literature, a graph $G$ is said to be {\em divisible}
by an integer $t$, or {\em $t$-divisible}, if $G$ admits an isomorphic
factorization into $t$ parts, although the analogy with the
number-theoretic notion of divisibility is only syntactical. The
notion of $2$-divisibility has also been termed \emph{bisectable},
with some authors tagging on the extra condition that the resulting
factors were also connected graphs.

The earliest work concerned the divisibility of the complete
graph. Extending a partial result of Guidotti \cite{Guidotti}, Harary,
Robinson, and Wormald \cite{iso-1} proved that the complete graph
$K_n$ is divisible by any integer $t$ which satisfies the obvious
necessary condition $t \mid \binom{n}{2}$. Most other existing
research on divisibility concentrates on trees and forests, perhaps
because their simple structure appears more tractable.
Algorithmically, Graham and Robinson proved in \cite{iso-9} that it is
NP-hard to decide whether a tree is $2$-divisible, while Harary and
Robinson \cite{iso-8} discovered a polynomial-time algorithm to decide
whether a tree admits a isomorphic factorization into two connected
graphs. The best general result on trees is due to Alon, Caro, and
Krasikov \cite{ACK}, who showed that every $m$-edge tree can be made
$2$-divisible by deleting only $O(m / \log \log m)$ edges.

%It was even conjectured by Harary and Robinson \cite{iso-8} that all
%trees with maximum degree 3 and even numbers of edges were
%bisectable, apart from two specific exceptional graphs of small
%order. Although this was disproved by Heinrich and Hor\'ak
%\cite{HH-disprove}, later work of Gvozdjak \cite{Gvozdjak} affirmed
%this intuition by showing that every tree's distance from
%bisectability (measured in terms of the minimum number of deletions
%sufficient to create a bisectable graph) is bounded by an absolute
%constant.

Once one considers general graphs, however, it becomes essentially
impossible to hope for $2$-divisibility or even closeness to
$2$-divisibility. It is therefore natural to ask what is the largest
2-divisible subgraph which must exist in a given graph. This problem
(stated below in generality for hypergraphs) was originally raised
independently by Jacobson and Sch\"onheim.

\begin{PROBLEM}
  Let the \textbf{self-similarity} of an $r$-uniform hypergraph $G$,
  denoted $\ISO(G)$, be the largest integer $s$ for which $G$
  contains a pair of edge-disjoint isomorphic sub-hypergraphs with $s$
  edges each.  For each positive integer $m$, let $\ISO_r(m)$ be the
  minimum of $\ISO(G)$ over all $r$-uniform hypergraphs with $m$
  edges.  Determine $\ISO_r(m)$.
\end{PROBLEM}

\noindent \textbf{Remark.} This paper focuses on graphs ($r=2$), so we
will write $\ISO(m)$ instead of $\ISO_2(m)$ throughout.

\medskip

%For the graph case ($r=2$), a simple argument shows that $\ISO_2(m)
%\geq \Omega( \sqrt{m} )$. Indeed, if there exists a vertex $v$ of
%degree at least $\sqrt{m}$, then the subgraph consisting of $v$ and
%its neighbors, with only the $d(v) \geq \sqrt{m}$ edges between $v$
%and $N(v)$, is a star, hence $2$-divisible except possible for only
%one edge. On the other hand, if every vertex has degree less than
%$\sqrt{m}$, then one can greedily construct a matching of size $m /
%(2 \sqrt{m})$, which is also $2$-divisible except possibly for only
%one edge.
The first main result in this area was due to Erd\H{o}s, Pach, and
Pyber \cite{EPP}. Specifically, they proved that there were absolute
constants $c_r$ and $C_r$ for which
\begin{displaymath}
  c_r m^{2/(2r-1)} \leq \ISO_r(m) \leq C_r m^{2/(r+1)} \cdot \frac{\log m}{\log \log m} \,.
\end{displaymath}
Their upper bound construction is based on an appropriately chosen
random $r$-uniform hypergraph. For graphs ($r=2)$, the powers of $m$
coincide at $m^{2/3}$, so their lower bound deviated only by a
logarithmic factor from their upper bound construction, which was
essentially the Erd\H{o}s-R\'enyi random graph.  At around the same
time, similar results were obtained independently by Alon and Krasikov
(unpublished), and by Gould and R\"odl. The latter group determined in \cite{GR} that
for 3-uniform hypergraphs, $\ISO_3(m) \geq \frac{1}{23} \sqrt{m}$,
which matched the upper bound exponent, but again fell short by a
logarithmic factor.  Very recently, Horn, Koubek, and R\"odl
\cite{HKR} announced lower bounds for $\ISO_4(m)$, $\ISO_5(m)$, and
$\ISO_6(m)$ which also came within poly-logarithmic factors of the
corresponding upper bounds derived from random hypergraphs.

The main result of our paper completely solves the graph case,
determining the asymptotic rate of growth of $\ISO(m) = \ISO_2(m)$.

\begin{THM}
  \label{thm:main}
  There are absolute constants $c$ and $C$ for which
  \begin{displaymath}
    c (m \log m)^{2/3}
    <
    \ISO(m)
    <
    C (m \log m)^{2/3} \,.
  \end{displaymath}
\end{THM}

The key idea is to exploit rare large deviations events through a
constructive algorithm, rather than to attempt to erase them with
union bounds. Incidentally, our upper bound construction is still
based on a random graph, but with a slightly modified edge
probability.

Inspired by the asymptotic optimality of random graphs in the problem
of Jacobson and Sch\"onheim, our next result explicitly studies the
self-similarity of random graphs.  The Erd\H{o}s-R\'enyi
random graph $\Gnp$ is constructed on the vertex set $[n] = \{1,
\ldots, n\}$ by taking each potential edge independently with
probability $p$.  We say that $\Gnp$ possesses a graph property
$\mathcal{P}$ {\em asymptotically almost surely}, or a.a.s.~for
brevity, if the probability that $\Gnp$ possesses $\mathcal{P}$ tends
to $1$ as $n$ grows to infinity. Since its first appearance in the
1960's, this beautiful object has been a central topic of study in
graph theory.  Surprisingly, many problems about random graphs arose
from research in various other areas of mathematics and theoretical
computer science. Yet despite the great amount of work devoted to this
topic over the past fifty years, many interesting unresolved questions
still remain to be answered.  For more on random graphs, we refer the
reader to the books \cite{Bollobas, JaLuRu}.

When $p < \frac{0.99}{n}$, it is well known that a.a.s.~all connected
components of $\Gnp$ are either trees or unicyclic (are trees with a
single additional edge).  Applying the previously mentioned
result of Alon, Caro, and Krasikov, or even Proposition
\ref{prop:starforest} below, it is then easy to see that the
self-similarity of $\Gnp$ in that regime is $\Theta(m)$ a.a.s., where
$m$ is the number of edges.  Our second result asymptotically
determines $\ISO(\Gnp)$ for the remaining range of $p$.

\begin{THM} \label{thm:randomgraph}
  \mbox{}

  \begin{enumerate}[(i)]
    \setlength{\itemsep}{1pt}
    \setlength{\parskip}{0pt} \setlength{\parsep}{0pt}
    % FIX: Po switched 1/n to 1/(2n).  Does it still work?
  \item If $\frac{1}{2n} \le p(n) \le \frac{1}{e^6}\sqrt{\frac{\log
        n}{n}}$, then $\ISO(\Gnp) = \Theta\left(n \cdot \frac{\log
        n}{\log \gamma(n)} \right)$ a.a.s., where $\gamma(n) =
    \frac{1}{p}\sqrt{\frac{\log n}{n}}$.
  \item If $p(n) > \frac{1}{e^6}\sqrt{\frac{\log n}{n}}$, then
    $\ISO(\Gnp) = \Theta(n^2p^2)$ a.a.s.
  \end{enumerate}
\end{THM}

We will prove this theorem in the next section.  Its proof illustrates the main ideas of the
argument for Theorem \ref{thm:main}, which follows in Section
\ref{sec:general}.

\medskip

\noindent \textbf{Notation.} Let $G$ be a graph with vertex set
$V$. For a subset of vertices $X \subset V$, let $G[X]$ be the
subgraph of $G$ induced by $X$. For a vertex $v \in V$, we use $N(v)$ 
to denote the set of neighbors of $v$. Given a bijection $f : V \rightarrow
V'$, let $f(G)$ be the graph with vertex set $V'$, where $x',y' \in
V'$ are adjacent if and only if there exist two adjacent vertices $x,y
\in V$ such that $f(x) = x'$ and $f(y) = y'$. For two graphs $G_1$ and
$G_2$ defined on the same vertex set, let $G_1 \cup G_2$ be the graph
obtained by taking the union of the edge sets of the two graphs, and
let $G_1 \cap G_2$ be the graph obtained by taking the intersection of
the edge sets of the two graphs.

The following standard asymptotic notation will be utilized
extensively.  For two functions $f(n)$ and $g(n)$, we write $f(n) =
o(g(n))$ if $\lim_{n \rightarrow \infty} f(n)/g(n) = 0$, and $f(n) =
O(g(n))$ or $g(n) = \Omega(f(n))$ if there exists a constant $M$ such
that $|f(n)| \leq M|g(n)|$ for all sufficiently large $n$.  We also
write $f(n) = \Theta(g(n))$ if both $f(n) = O(g(n))$ and $f(n) =
\Omega(g(n))$ are satisfied.  All logarithms will be in base $e
\approx 2.718$.

\section{Random graphs}
\label{sec:random}

We will use the following well-known concentration result, which is a
consequence of Theorems A.1.11 and A.1.13 in the book \cite{AlSp}. Let
$\bin(n,p)$ denote the binomial random variable with parameters $n$
and $p$.

% A.1.11: P[ X-pn >= a ]  < e^{ -a^2/(2pn) + a^3/(2(pn)^2) }
% A.1.13: P[ X-pn <= -a ] < e^{ -a^2/(2pn) }
%
% Lower tail is OK
% Upper tail: if a < (2/3)pn then
%   bound < e^{ -a^2/(2pn) + a^2/(3pn) }
%         = e^{ -a^2/(6pn) }
% and if a >= (2/3)pn then Alon-Spencer says on the same page that
%   bound < P[ X-pn > (2/3)pn ]
%         < e^{ -(2/27)pn }
% but a^2/(pn) could get as big as (pn), when a = pn
% so we need to say that all tails are at most
%           e^{ -(2/27) a^2/(pn) }

\begin{THM} \label{thm:chernoffinequality}
If $X \sim \bin(n,p)$ and $\lambda \leq np$, then
\begin{displaymath}
  \pr{ |X - np| \geq \lambda }
  \leq 
  e^{-\frac{\lambda^2}{15 np}} \,.
\end{displaymath}
\end{THM}

We begin by analyzing the self-similarity of random graphs.  In
addition to being an interesting question in its own right, this
investigation also suggests good intuition for general graphs.  The
upper bounds on $\ISO(\Gnp)$ follow from relatively straightforward
union bounds.

\begin{proof}[Proof of upper bound in Theorem \ref{thm:randomgraph}.]

  Suppose that we are seeking a pair of edge-disjoint isomorphic
  subgraphs with $t$ edges.  This task is equivalent to finding
  subgraphs $H'$ with $2t$ edges that can be partitioned into the
  union $H \cup \pi(H)$, for some $t$-edge subgraph $H$ and a
  permutation $\pi$ of the vertex set.  The expected number of such
  subgraphs $H'$ in $\Gnp$ is at most
  \begin{align} \label{eq:upperbound}
    {{n \choose 2} \choose t} \cdot n! \cdot p^{2t} 
    <
    \left(\frac{en^2p^2}{t}\right)^t e^{n\log n} \,,
  \end{align}
  where the first binomial coefficient counts the number of ways to select
  $t$ edges for $H$ out of all $\binom{n}{2}$ available, and the $n!$
  bounds the number of permutations $\pi$ of the vertex set.  Together,
  these choices determine the $2t$ edges which make up $H'$, which appear
  with probability $p^{2t}$.  Thus, if we select a value of $t$ for which
  the right hand side of \eqref{eq:upperbound} becomes $o(1)$, we will
  establish that the number of such $H'$ is zero a.a.s., and hence
  $\ISO(\Gnp) < t$ a.a.s.

  We separately specify suitable choices for $t$ for the two regimes
  of $p$ that we consider in this theorem.  For part (i), where
  $\frac{1}{2n} \le p \le \frac{1}{e^6}\sqrt{\frac{\log n}{n}}$, we
  use $t = \frac{n \log n}{\log \gamma}$, where $\gamma =
  \frac{1}{p} \sqrt{\frac{\log n}{n}}$.  Note that in this range we
  have $e^6 \le \gamma \le 2\sqrt{n\log n}$.  Then the right hand side
  of \eqref{eq:upperbound} becomes
  \begin{displaymath}
    \left(
      \frac{en^2\cdot \frac{\log n}{\gamma^2 n}}
      {\frac{n \log n}{\log \gamma}} 
    \right)^{\frac{n\log n}{\log \gamma}} e^{n\log n}
    = 
    \left(
      \frac{e \log \gamma}{\gamma^2}
    \right)^{\frac{n\log n}{\log \gamma}} e^{n\log n}
    = 
    e^{-\frac{n\log n}{\log \gamma} \cdot \log\left(\frac{\gamma^2}{e \log \gamma}\right)} \cdot e^{n\log n}.
  \end{displaymath}
  Since $\gamma \ge e^6$, we have $\log \left(\frac{\gamma^2}{e \log
      \gamma}\right) > \frac{3}{2}\log \gamma$, and hence the right
  hand side of \eqref{eq:upperbound} is at most $e^{-\frac{3}{2}n\log
    n}\cdot e^{n\log n} = o(1)$.

  For part (ii), where $p \ge \frac{1}{e^6}\sqrt{\frac{\log n}{n}}$,
  we specify $t = e^{12} n^2p^2$.  The right hand side of
  \eqref{eq:upperbound} then becomes
  \begin{displaymath}
    \left(
      \frac{1}{e^{11}} 
    \right)^{e^{12}n^2p^2} e^{n\log n}
    \leq 
    \left(
      \frac{1}{e^{11}} 
    \right)^{n\log n} e^{n\log n}  
    = 
    o(1) \,.
  \end{displaymath}
\end{proof}

The remainder of this section is devoted to constructing large
self-similar subgraphs in $\Gnp$.  The structure given in the
following definition turns out to be extremely useful (both for this
section and the next section).

\begin{DEF} 
  \label{def:batches} 
  Let $d$ and $k$ be positive integers.
  \begin{enumerate}[(i)]
    \setlength{\itemsep}{1pt} \setlength{\parskip}{0pt}
    \setlength{\parsep}{0pt}
  \item A \textbf{$\boldsymbol{d}$-star} is a graph consisting of
    $d+1$ vertices and $d$ edges, where one of the vertices has degree
    $d$. We sometimes simply refer to these graphs as \textbf{stars}.
  \item A \textbf{$\boldsymbol{(d,k)}$-star-forest} is a collection of
    $k$ vertex-disjoint $d$-stars.  We denote a $(d,k)$-star-forest by
    the set of pairs $\{(v,N_v) : v \in B\}$, where $B$ is a set of
    $k$ vertices, and for each $v$, the set $N_v \subset N(v)$ is a
    disjoint set of $d$ neighbors of $v$.
  \end{enumerate}
\end{DEF}

The following two propositions were the key ideas in \cite{EPP}.  We
include their proofs for completeness, as well as to illuminate the
points at which we introduce our new arguments.  The first claim
asserts that the self-similarity of a graph is large if there are many
non-isolated vertices.

\begin{PROP} \label{prop:starforest}
Let $G$ be a graph on $n$ vertices with no isolated vertices.
Then $\ISO(G) \ge \frac{n-2}{4}$.
\end{PROP}
\begin{proof}
  We first prove that $G$ contains vertex-disjoint stars that cover
  all the vertices of the graph. Given a graph $G$, iteratively remove
  edges that connect two vertices of degree at least two (in an
  arbitrary order). Clearly, this process never creates isolated
  vertices, and the final graph consists only of stars because all
  remaining vertices of degree two or more are non-adjacent.

  It remains to show that any $n$-vertex star forest contains two
  large edge-disjoint isomorphic subgraphs $G_1$ and $G_2$. We
  consider the stars in the forest by their type.  Note that 1-stars
  are nothing more than single edges, so for every two 1-stars, we can
  put one of them in $G_1$ and the other in $G_2$.  We account for
  this as a contribution of $+1$ toward $\ISO(G)$ from the four
  vertices in the two 1-stars.  On the other hand, for $d \ge 2$, we
  can split the edges of every $d$-star into two sets of size
  $\left\lfloor \frac{d}{2} \right\rfloor$, possibly with one edge
  left over.  By adding one part to $G_1$ and the other to $G_2$, we
  see that the $d+1$ vertices of each $d$-star contribute $+
  \left\lfloor \frac{d}{2} \right\rfloor$ to $\ISO(G)$.  Accumulating
  the contributions from all vertices, except possibly for at most two
  vertices from a single unpaired $1$-star, we find that
  \begin{displaymath}
    \ISO(G)
    \ge 
    (n-2) \cdot \min \left\{
      \frac{1}{4}, \,\, \min_{d \ge 2} \left\{ 
        \frac{\lfloor d/2 \rfloor}{d+1} 
      \right\} 
    \right\} 
    = 
    \frac{n-2}{4} \,.
  \end{displaymath}
\end{proof}

Although our problem considers the self-similarity within a single
graph, our lower bound argument first separates the given graph into
two disjoint subgraphs, and constructs a suitable mapping between
them which overlaps many edges.

\begin{DEF}
  \label{def:map_gives_iso}
  Let $G_1$ and $G_2$ be two edge-disjoint graphs, on possibly
  overlapping vertex sets $V_1$ and $V_2$ of the same cardinality.
  Let their \textbf{similarity} $\ISO(G_1, G_2)$ be the maximum
  integer $s$ such that there exists a bijection $f : V_1 \rightarrow
  V_2$ for which $f(G_1) \cap G_2$ contains $s$ edges.
\end{DEF}

% \begin{PROP} 
% For $i=1,2$, let $G_i$ be a graph over the vertex set $V_i$, and suppose
% that $G_1$ and $G_2$ are edge-disjoint. Let $f: V_1 \rightarrow V_2$
% be an injective map, and suppose that $f(G_1) \cap G_2$ contains
% $k$ edges. Then $\ISO(G_1,G_2) \ge k$.
% \end{PROP}
% \begin{proof}
% Let $H_1 = G_1 \cap f^{-1}(G_2)$ and $H_2 = f(G_1) \cap G_2$. By definition,
% these graphs have sizes at least $k$, and the map
% $f$ shows that $H_1$ and $H_2$ are isomorphic.
% \end{proof}

The next proposition uses a random mapping as the input in
Definition \ref{def:map_gives_iso}, in order to measure similarity
of two random bipartite graphs.

\begin{PROP} \label{prop:pair-volume} For $i=1,2$, let $G_i$ be
  edge-disjoint bipartite graphs with parts $A_i$ and $B_i$, where
  $|A_1| = |A_2| = n_1$ and $|B_1| = |B_2| = n_2$.  Suppose that $A_1
  \cup A_2$ and $B_1 \cup B_2$ are disjoint, but $A_1$ may intersect
  $A_2$ and $B_1$ may intersect $B_2$.  Then $\ISO(G_1, G_2) \ge
  \frac{|E(G_1)| |E(G_2)|}{n_1 n_2}$.
\end{PROP}
\begin{proof}
  Independently sample uniformly random bijections from $A_1$ to $A_2$ and
  from $B_1$ to $B_2$, and let $f$ be their combination. For each pair of
  edges $e_1 \in E(G_1)$ and $e_2 \in E(G_2)$, the probability that $e_1$
  gets mapped to $e_2$ by $f$ is exactly $\frac{1}{n_1 n_2}$.  Such a
  situation contributes $+1$ to the intersection size $f(G_1) \cap G_2$.
  Therefore, by linearity of expectation, the expected number of edges in
  $f(G_1) \cap G_2$ is at least $\frac{|E(G_1)| |E(G_2)|}{n_1 n_2}$, and
  there exists a suitable $f$ which achieves that bound.
\end{proof}

\begin{COR} \label{cor:volume} Let $G$ be a bipartite graph with parts
  $A$ and $B$ such that $|E(G)| \ge 10$.  Then $\ISO(G) \ge
  \frac{|E(G)|^2}{5 |A| |B|}$.
\end{COR}
\begin{proof}
  Arbitrarily partition $G$ into two edge-disjoint subgraphs $G_1 \cup
  G_2$ with $\lfloor \frac{1}{2} |E(G)| \rfloor \ge \frac{|E(G)|-1}{2} \ge \frac{9|E(G)|}{20}$ edges, and apply 
Proposition
  \ref{prop:pair-volume}.
\end{proof}

\begin{COR} \label{cor:general-volume} Let $G$ be a graph with $n$
  vertices and $m$ edges, where $m \ge 20$.  Then
  $\ISO(G) \ge \frac{m^2}{5 n^2}$.
\end{COR}
\begin{proof}
  Let $A \cup B$ be a bipartition of the vertex set of $G$ chosen 
  uniformly at random.
  The probability of a single edge intersecting both parts is exactly $\frac{1}{2}$,
  and thus by averaging, there exists a bipartition $A \cup B$ for which
  the bipartite graph $H$ between $A$ and $B$ contains at least $\frac{m}{2}$ edges.
  Since $|A||B| \le \frac{n^2}{4}$ and $m/2\geq 10$, by Corollary \ref{cor:volume}, we have
  $\ISO(G) \ge \frac{(m/2)^2}{5(n^2/4)}=\frac{m^2}{5 n^2}$.
\end{proof}

To prove Proposition \ref{prop:pair-volume}, we considered a random
bijection between the two vertex sets, as there exists a map such that
the resulting number of overlapping edges is at least its
expectation. This strategy turns out to be strong enough when the
graph is dense.  On the other hand, for sparse graphs, Proposition
\ref{prop:starforest} produces a reasonable bound. These were the key
steps used by Erd\H{o}s, Pach, and Pyber in \cite{EPP}. In order to
establish Theorem \ref{thm:randomgraph}, however, we need something
slightly more powerful for the intermediate edge density regime.

The key new ingredient is to design a vertex permutation that performs
better than a uniformly random one. To sketch our argument, consider
the illustrative case $p = n^{-1/2}$, which represents the most
delicate range. We first randomly split the vertices into four parts
$A_1, A_2, B_1, B_2$ of equal size, and let $G_i$ be the bipartite
graph formed by the edges between $A_i$ and $B_i$.  We discard all
other edges, and bound only the similarity between $G_1$ and $G_2$.
Rather than searching for an unstructured permutation of the whole
vertex set, we build a favorable bijection $f: A_1 \cup B_1
\rightarrow A_2 \cup B_2$ which sends $A_1$ to $A_2$ and $B_1$ to
$B_2$ with many overlapping edges.
%\begin{center}\end{center}, and then extend arbitrarily to a full permutation.  
Note that if we let $f$ be a uniformly random
bijection from $A_1 \cup B_1$ to $A_2 \cup B_2$, then we essentially
recover Proposition \ref{prop:pair-volume}, thus producing a lower
bound of order only $\Theta(n)$, which falls short of Theorem
\ref{thm:randomgraph} by a logarithmic factor.

We start with a uniformly random bijection from $B_1$ to $B_2$, and
carefully extend it from $A_1$ to $A_2$ as follows. Consider a fixed
vertex $v_1$ in $A_1$ and a fixed vertex $v_2$ in $A_2$.  If we mapped
$v_1$ to $v_2$, we would increase the number of overlapping edges by
exactly $|f(N(v_1)) \cap N(v_2)|$, where $N(v_i)$ represents the set
of neighbors of $v_i$ in $B_i$.  (Recall that we discarded all other
edges, so the $v_i$ only have neighbors in their corresponding $B_i$.)
Since we have $p = n^{-1/2}$, if $v_2$ is chosen uniformly at random,
the expected size of the set $f(N(v_1)) \cap N(v_2)$ is some constant
$\lambda$, and this observation led to the $\Theta(n)$ lower bound
when considering a uniformly random bijection.

The crucial observation is that for each individual pair of $v_i$, the
overlap $|f(N(v_1)) \cap N(v_2)|$ asymptotically has the Poisson
distribution with mean $\lambda$. Therefore, with probability at least
$n^{-\varepsilon}$, it will be of size at least $\varepsilon' \frac{\log
  n}{\log \log n}$ for some small constants $\varepsilon$ and
$\varepsilon'$. Since $A_2$ has $\frac{n}{4}$ vertices, the expected
number of vertices $v_2 \in A_2$ that will give this high gain
together with $v_1$ is $\Omega(n^{1-\varepsilon})$. In particular, it is
very likely that there exists a suitable vertex $v_2$ for $v_1$ such
that $|f(N(v_1)) \cap N(v_2)| \geq \varepsilon' \frac{\log n}{\log \log
  n}$, and we will map $v_1$ to $v_2$ in such a situation. By
repeating this for a constant proportion of vertices in $A_1$, we will
obtain $\ISO(\Gnp) \ge \Omega(n \cdot \frac{\log n}{\log \log
  n})$. Since $\gamma = \sqrt{\log n}$, this gives 
$\ISO(\Gnp) \ge \Omega( n \cdot \frac{\log n}{\log \gamma(n)})$ for our
choice of $p$. Our next two lemmas formalize this intuition.

\begin{LEMMA} 
  \label{lem:poisson_random} 
  Let $n$ and $p$ satisfy $n^{-\frac{21}{40}} \le p \le
  \frac{1}{e^6}\sqrt{\frac{\log n}{n}}$, and define $\gamma =
  \frac{1}{p}\sqrt{\frac{\log n}{n}}$.  Let $N_1, \ldots, N_s \subset
  B$ be $s \ge n^{1/3}$ disjoint sets of size $\frac{np}{16}$, and
  consider the random set $B_p$, where we take each element of $B$
  independently with probability $p$.  Then with probability at least
  $1 - e^{-\Omega(n^{1/12})}$, there is an index $i$ such that $|B_p
  \cap N_i| \ge \frac{\log n}{20\log \gamma}$.
\end{LEMMA}
\begin{proof}
  Let $t = \left\lceil \frac{\log n}{20\log \gamma} \right\rceil$.  In
  our range of $p$, we always have $2 \le t \le \big\lceil \frac{\log
    n}{120} \big\rceil$, so in particular $t \le \frac{\log n}{10\log
    \gamma}$.  For a fixed index $i$, the probability that $|B_p \cap
  N_i| \ge \frac{\log n}{20 \log \gamma}$ is at least
  $\binom{|N_i|}{t} p^t (1-p)^{|N_i|-t}$.  Using the bounds
  $\binom{n}{k} \geq \big( \frac{n}{k} \big)^k$ and $(1-p) \geq
  e^{-\frac{16}{15} p}$ for small $p$, we have
    \begin{align*}
      \binom{|N_i|}{t} p^t (1-p)^{|N_i|-t}
      &\ge \left( \frac{np^2}{16t} \right)^t e^{- np^2/15}
      = \left( \frac{\log n}{16\gamma^2 t} \right)^t e^{- np^2/15} \\
      &\ge \left( \frac{10\log \gamma}{16\gamma^2} \right)^{\log n/(10\log \gamma)} \cdot n^{-1/(15e^{12})} \\
      &= e^{-\frac{\log n}{10 \log \gamma} \cdot \log\left( \frac{16\gamma^2}{10\log \gamma}\right)} \cdot n^{-1/(15e^{12})} \,,
    \end{align*}
    which by $\log\left( \frac{16\gamma^2}{10\log \gamma}\right) \le 2 \log \gamma$ (deduced from $\gamma \ge e^6$), is at least
    \begin{align*}
      & e^{-\frac{\log n}{5}} \cdot n^{-1/(15e^{12})} \ge n^{-1/4} \,.
    \end{align*}
    Hence the expected number of indices $i$ such that $|B_p \cap N_i|
    \ge \frac{\log n}{20 \log \gamma}$ is at least $s \cdot n^{-1/4}
    \ge n^{1/12}$.  Since the sets $N_i$ are disjoint, the above
    events for different choices of $i$ are mutually independent.
    Therefore, by Chernoff's inequality, with probability at least $1
    - e^{-\Omega(n^{1/12})}$, we can find an index $i$ (indeed,
    several) for which $|B_p \cap N_i| \ge \frac{\log n}{20 \log
      \gamma}$.
\end{proof}

The previous estimate enables us to bound the similarity between
random bipartite graphs.

\begin{LEMMA} \label{lem:mapping_random}
%There exist a positive constant $\varepsilon$ such that the following holds.
  Let $n$ and $p$ satisfy $n^{-\frac{21}{40}} \le p \le
  \frac{1}{e^6}\sqrt{\frac{\log n}{n}}$, and let $\gamma =
  \frac{1}{p}\sqrt{\frac{\log n}{n}}$.  Let $A_1,B_1,A_2,B_2$ be
  disjoint sets of size $\frac{n}{4}$ each, and for each $i = 1,2$,
  let $G_i$ be a random bipartite graph with parts $A_i$ and $B_i$,
  where each edge appears independently with probability $p$.  Then
  $\ISO(G_1,G_2) \ge \frac{n\log n}{160\log \gamma}$ a.a.s.
\end{LEMMA}
%\noindent \textbf{Remark.}  In Proposition \ref{prop:pair-volume},
%each vertex contributes to the similarity by an average of $np^2 =
%\frac{\log n}{\gamma^2}$, where $p$ and $\gamma$ are defined as
%above. In contrast, in Lemma \ref{lem:mapping_random}, each vertex
%contributes to the similarity by about $\frac{\log n}{\log \gamma}$,
%producing even a \emph{polynomial}\/ gain when $\gamma$ is a power of
%$n$.
\begin{proof}%[Proof of Lemma \ref{lem:mapping_random}.]
  Start with a uniformly random bijection $f$ from $B_1$ to $B_2$, and also
  expose all edges in the random bipartite graph $G_2$. Since $p \ge
  n^{-\frac{21}{40}}$, Chernoff's inequality and a union bound establish
  that a.a.s., all degrees in $G_2$ are between $\frac{np}{8}$ and $np$.
  Condition on this event.  We expose the edges in the bipartite graph
  $G_1$ by iterating over the vertices in $A_1$, exposing each vertex's
  incident edges in turn.  Consider the following greedy algorithm for
  finding a bijection between $A_1$ and $A_2$.  Let $A_1'$ be the set of
  vertices in $A_1$ whose edges have been exposed, and suppose that we have
  an injective map $f: A_1' \rightarrow A_2$ such that for all $x \in
  A_1'$, $f(N(x))$ and $N(f(x))$ intersect in at least $\frac{\log
  n}{20\log \gamma}$ vertices. Let $A_2' = f(A_1')$, and let $A_i'' = A_i
  \setminus A_i'$ for $i=1,2$.  Suppose that $|A_1''| = |A_2''| \ge
  \frac{|A_1|}{2}$ at some point of the process.

  We first prove that the graph $A_2'' \cup B_2$ contains a
  $(\frac{np}{16}, n^{1/3})$-star-forest. Indeed, let $k$ be the largest
  integer such that there exists a $(\frac{np}{16}, k)$-star-forest
  $\{(x,N_x) : x \in X\}$ for some set $X \subset A_2''$ of size $|X| = k$,
  and suppose that $k < n^{1/3}$. Let $N(X)$ be the union of all
  neighborhoods of vertices in $X$.  We know that for every vertex $w \in
  A_2'' \setminus X$, we have $|N(w) \cap N(X)| \ge \left(\frac{1}{8} -
  \frac{1}{16}\right)np \ge \frac{np}{16}$ as otherwise we find a
  $(\frac{np}{16}, k+1)$-star-forest, contradicting maximality. Therefore,
  there are at least $\frac{np}{16} \cdot (|A_2''| - |X|) \ge
  \frac{n^2p}{128}$ edges between the sets $A_2'' \setminus X$ and $N(X)$,
  and in particular, the set $N(X)$ has at least $\frac{n^2p}{128}$
  incident edges in $G_2$.  Note that $|N(X)| \le knp \le n^{4/3}p$, 
  since we conditioned on all degrees in $G_2$ being at most $np$, and by
  the same reason, the
  number of edges incident to $N(X)$ must be at most $n^{7/3} p^2 <
  \frac{n^2p}{128}$, contradiction.  Therefore, we have $k \ge n^{1/3}$, as
  claimed.

  Now take any vertex $v_1 \in A_1''$, and expose its edges to $B_1$.
  Its neighborhood $N(v_1)$ is a random subset of $B_1$, where each
  vertex of $B_1$ appears independently with probability $p$.  Since
  the bijection $f : B_1 \rightarrow B_2$ was fixed from the outset,
  the image of the neighborhood $f(N(v_1))$ is also a random subset of
  $B_2$ with the same product distribution.  By Lemma
  \ref{lem:poisson_random}, with probability at least $1 -
  e^{-\Omega(n^{1/12})}$, we can find a vertex $v_2 \in X \subset
  A_2''$ such that $|f(N(v_1)) \cap N_{v_2}| \ge \frac{\log n}{20 \log
    \gamma}$, where $X$ and $N_{v_2}$ were from the star forest
  constructed above.  Define $f(v_1) = v_2$ and repeat the procedure.
  Since the probability of success at each round is $1-o(n^{-1})$, we
  can successfully iterate $\frac{|A_1|}{2}$ times a.a.s., and then
  finish by extending $f$ by an arbitrary bijection between the
  non-mapped vertices of $A_1$ and $A_2$.  In this way, we obtain a
  bijection $f$ such that the number of edges in $f(G_1) \cap G_2$ is
  at least $\frac{|A_1|}{2} \cdot \frac{\log n}{20\log \gamma} =
  \frac{n \log n}{160\log \gamma}$, as desired.
\end{proof}
% [FIX REMARK: By considering several vertices at a time, we might be
% able to improve the range of $p$ into $n^{-1+\varepsilon}$. This
% might be interesting for hypergraphs.]

We are now ready to prove the lower bounds of Theorem \ref{thm:randomgraph}.

\begin{proof}[Proof of lower bound in Theorem \ref{thm:randomgraph}.]
  Part (i) has two subcases. First, for $\frac{1}{2n} \le p \le
  n^{-21/40}$, note that $\gamma = \frac{1}{p} \sqrt{\frac{\log n}{n}}
  \geq n^{1/40} \sqrt{\log n}$, so the desired lower bound is of order
  $n \cdot \frac{\log n}{\log \gamma} = \Theta(n)$.  In this range,
  the number of non-isolated vertices is $\Theta(n)$ a.a.s., so
  Proposition \ref{prop:starforest} completes this case.  For the next
  range $n^{-\frac{21}{40}} \le p \le \frac{1}{e^6}\sqrt{\frac{\log
      n}{n}}$, we apply Lemma \ref{lem:mapping_random} after splitting
  the vertex set into four parts.  Part (ii) follows directly from
  Corollary \ref{cor:volume}.
\end{proof}

\section{Self-similarity of general graphs}
\label{sec:general}

Although general graphs are not intrinsically random, we apply
probabilistic techniques to find large edge-disjoint isomorphic
subgraphs.  
The outline of our proof for general graphs is similar to that for
random graphs (see the discussion following Corollary \ref{cor:general-volume}
in the previous section).  The key idea is to exploit tail events in
the Poisson distribution.  However, establishing this was somewhat
easier for random graphs since we had independence, and could expose
edges in a controlled manner. For general graphs, there are no random
edges to expose.  Instead, we turn to star forests, which were also an
important component in the proof of Lemma \ref{lem:mapping_random}.

Let $G$ be a given graph on $n$ vertices with average degree $d$. As
before, we begin by randomly splitting the vertices into four parts
$A_1, A_2, B_1, B_2$, and consider the bipartite graphs $G_i$ formed
by the edges between $A_i$ and $B_i$.  We attempt to find a total of
$\Omega(n^{1-\alpha})$ many $(\frac{d}{8},n^\alpha)$-star-forests
$S_{i,j} = \{(v,N_v) : v \in X_{i,j}\}$ for $i=1,2$, $1 \le j \le
\Omega(n^{1-\alpha})$, where the sets $X_{i,j} \subset A_i$ are
disjoint for different indices. Note that $\bigcup X_{i,j}$ then cover
a constant fraction of each $A_i$, and hence the edges in these star
forests constitute a constant fraction of the edges in the entire
graph $G$.  If we fail to find such star forests, then we will be able
to pass to a subgraph where we can find even larger isomorphic
subgraphs. On the other hand, once we find such star forests, we take
a random bijection $f_B$ from $B_1$ to $B_2$, and extend it by
independent bijections from $X_{1,j}$ to $X_{2,j}$.  To this end, we
declare $f_B$ to be \emph{good}\/ for the index $j$ if it can be
extended to a bijection between the sets $B_1 \cup X_{1,j}$ and $B_2
\cup X_{2,j}$ so that the two star-forests overlap in
$\Omega\left(|X_{1,j}| \cdot \frac{\log n}{\log \left(\frac{n \log
        n}{d^2} \right)}\right)$ edges under the map. If some
bijection $f_B$ happens to be good for a constant proportion of
indices $j$, then we can extend the bijection $f_B$ to the sets
$X_{1,j}$ for these indices, and thereby construct a map $f$ that
overlaps many edges of $G_1$ and $G_2$.

To begin this program, our first lemma establishes the tail
probability of the main random variable in our setting.  It is the
analogue of Lemma \ref{lem:poisson_random}.

\begin{LEMMA} 
  \label{lem:poisson_general} 
  Let $\alpha < \frac{1}{2}$ be a fixed positive real number, and let
  $d$ and $n$ satisfy $n^{\frac{1}{2} - \frac{\alpha}{16}} \le d \le
  \sqrt{\alpha n \log n}$. Let $N_1, \ldots, N_s \subset [n]$ be fixed
  disjoint sets of size $\frac{d}{2}$ for some $s \ge \frac{1}{5} n^{\alpha}$, and
  let $N$ be a uniformly random subset of $[n]$ with exactly $d$
  elements.  Then with probability at least $1 -
  e^{-\Omega(n^{\alpha/4})}$, there exists an index $i$ such that $|N
  \cap N_i| \ge \frac{\alpha \log n}{8 \log \left(\frac{n \log n}{d^2}
    \right)}$.
\end{LEMMA}

\begin{proof}
  Let $N'$ be a random subset of $[n]$ obtained by independently
  taking each element with probability $\frac{d}{2n}$. The distribution
  of $N'$ conditioned on the event $|N'|\le d$ can be coupled with the
  random variable $N$, so that $N' \subset N$ (given $N'$, let $N$ be
  a set of size $d$ containing $N'$ chosen uniformly at random). 
  By Chernoff's bound,
  the probability of $|N'|>d$ is at most $e^{-\Omega(d)} <
  e^{-\Omega(n^{\alpha/4})}$, since $d \geq n^{\frac{1}{2} -
    \frac{\alpha}{16}}$ and $\alpha < \frac{1}{2}$.
  %As $|N'|$ is
  %distributed binomially with mean $\frac{d}{2}$, the Chernoff bound
  %together with a standard coupling argument shows that $N' \subset N$
  %except on a rare event of probability at most $e^{-\Omega(d)} <
  %e^{-\Omega(n^{\alpha/4})}$, since $d \geq n^{\frac{1}{2} -
  %  \frac{\alpha}{16}}$ and $\alpha < \frac{1}{2}$.  
  Therefore, in order to prove our lemma, 
  it suffices to show that with probability at least $1 -
  e^{-\Omega(n^{\alpha/4})}$, there exists an index $i$ such that $|N'
  \cap N_i| \geq \frac{\alpha \log n}{8 \log \left(\frac{n \log n}{d^2}
    \right)}$.  Define
  \begin{displaymath}
    \gamma = \frac{n \log n}{d^2} 
    \quad \text{and} \quad
    t = \left\lceil
      \frac{\alpha \log n}{8 \log \gamma} 
    \right\rceil \,.
  \end{displaymath}
  Since $n^{\frac{1}{2} - \frac{\alpha}{16}} \le d \le \sqrt{\alpha n
    \log n}$, we have
  \begin{align}
    \nonumber
    2 < \frac{1}{\alpha} &\leq \gamma \leq
    n^{\frac{\alpha}{8}} \log n,
  \end{align}
  from which it follows that
  \begin{align}
    %\label{ineq:unrounded-t}
    t \ge \frac{\alpha \log n}{8 \log \gamma} &\ge
    \frac{\alpha \log n}{8 \log (n^{\frac{\alpha}{8}} \log n)} =
    \frac{\alpha \log n}{\alpha \log n + 8 \log \log n} \ge
    \frac{1}{2}
     %\leq \frac{\alpha \log n}{4} 
    \,,
  \end{align}
  for sufficiently large $n$. %the bound in \eqref{ineq:unrounded-t} surpasses $\frac{1}{2}$.  
  Therefore, the
  rounding effect in the definition of $t$ at most doubles the value,
  and we have $1 \leq t \leq \frac{\alpha \log n}{4 \log \gamma}$.
  
  For each index $i$, let $E_i$ be the event that $|N' \cap N_i| \geq t$.
  As $|N' \cap N_i|$ is binomially distributed, just as in the proof of
  Lemma \ref{lem:poisson_random}, we may use the bounds $\binom{n}{k} \geq
  \big( \frac{n}{k} \big)^k$ and $1-p > e^{-2p}$ (for small $p$) to find
  \begin{align*}
    \pr{E_i} \geq \binom{|N_i|}{t} \left( \frac{d}{2n} \right)^t
    \left( 1 - \frac{d}{2n} \right)^{|N_i| - t}
    \geq \left( \frac{d/2}{t} \right)^t \left( \frac{d}{2n} \right)^t
    \left( e^{-\frac{d}{n}} \right)^{\frac{d}{2}}
    = \left( \frac{d^2}{4nt} \right)^t e^{-\frac{d^2}{2n}}.
  \end{align*}
  Substitute $t \le \frac{\alpha \log n}{4\log \gamma}$ to get
  \begin{align*}
    \pr{E_i} &\ge 
    \left( \frac{d^2}{4n} \cdot \frac{4 \log \gamma}{\alpha \log n} \right)^t 
    e^{-\frac{d^2}{2n}} 
    = \left( \frac{\log \gamma}{\alpha \gamma} \right)^t e^{-\frac{\log n}{2 \gamma}} \,.
  \end{align*}
  Since $\alpha < \frac{1}{2}$, $\log \gamma > \log 2$, and $t \le \frac{\alpha \log n}{4\log \gamma}$, 
  this is at least
  \begin{displaymath}
    \left( \frac{1}{\gamma} \right)^{\frac{\alpha \log n}{4 \log \gamma}}
    e^{-\frac{\log n}{2\gamma}}
    =
    n^{-\frac{\alpha}{4}} n^{-\frac{1}{2\gamma}}
    \geq
    n^{-\frac{\alpha}{4}} n^{-\frac{\alpha}{2}} = n^{-\frac{3\alpha}{4}} \,.
  \end{displaymath}
  The $E_i$ are independent because the $N_i$ are disjoint.
  Therefore the number of $E_i$ that occur stochastically
  dominates a binomial random variable with mean $sn^{-3\alpha/4} \ge \frac{1}{5} n^{\alpha/4}$, and
  we conclude by the Chernoff bound that at least one $E_i$ (indeed,
  several) occurs with probability $1 - e^{-\Omega(n^{\alpha/4})}$, as
  desired.
\end{proof}

In the previous section, in Lemma \ref{lem:mapping_random}, 
we exploited the fact that the given graph
was random and the edges were independent. This trick is too
restrictive to be applied to general graphs. However, the next lemma
says that for star-forests, one can obtain a lemma similar to Lemma
\ref{lem:mapping_random}.

\begin{LEMMA} 
  \label{lem:general_mapping} 
  Let $\alpha < \frac{1}{2}$ be a fixed positive real number, and
  suppose that $n$ and $d$ satisfy $n^{\frac{1}{2} -
    \frac{\alpha}{16}} \le d \le \sqrt{\alpha n \log n}$, and are
  sufficiently large.  For $i=1,2$, let $G_i$ be a $(d,
  n^{\alpha})$-star-forest $\{(v,N_v): v \in X_i\}$ in the
  vertex set $X_i \cup B_i$, where $|X_i| = n^{\alpha}$ and $|B_i| =
  n$. The bijection $f_B$ from $B_1$ to $B_2$ chosen uniformly at
  random satisfies the following property with probability at least $1
  - e^{-\Omega(n^{\alpha/4})}$: $f_B$ can be extended to $X_1\cup B_1$
  so that the graph $f_B(G_1) \cap G_2$ has at least $|X_1| \cdot
  \frac{\alpha \log n}{36 \log \left( \frac{n \log n}{d^2} \right)}$
  edges.
\end{LEMMA}
\begin{proof}
  Consider a uniformly random bijection $f_B$ from $B_1$ to $B_2$.  As
  in the proof of Lemma \ref{lem:mapping_random}, we will pick
  vertices of $X_1$ one at a time, mapping each one to some vertex in
  $X_2$ in such a way that their neighbors intersect in at least
  $\frac{\alpha \log n}{9 \log \left( \frac{n \log n}{d^2} \right)}$
  vertices under the map $f_B$.  By repeating this for $|X_1|/4$
  steps, we then extend $f_B$ to form a total of at least
  $\frac{|X_1|}{4} \cdot \frac{\alpha \log n}{9 \log \left(
      \frac{n \log n}{d^2} \right)}$ overlapping edges, as required.

  To this end, suppose that we have already embedded some set $X_1'
  \subset X_1$ of size less than $|X_1|/4$, and let $X_2'$ be the
  image of $X_1'$.  Further suppose that we have only exposed the
  outcome of $f_B$ on the neighbors of $X_1'$. Let $B_1' = \bigcup_{x
    \in X_1'} N_x$ and $B_2'$ be its image (which is already fully
  determined by our partial exposure).  The unexposed remainder of
  $f_B$, conditioned on the previous outcome, is a random uniform
  bijection from $B_1 \setminus B_1'$ to $B_2 \setminus B_2'$.  Choose
  an arbitrary vertex $x_1 \in X_1 \setminus X_1'$.  Call a vertex
  $x_2 \in X_2 \setminus X_2'$ \emph{available}\/ if $|N_{x_2}
  \setminus B_2'| \ge \frac{d}{2}$, or equivalently, $|N_{x_2} \cap
  B_2'| \le \frac{d}{2}$.  Since each unavailable vertex accounts for
  at least $\frac{d}{2}$ vertices of $|B_2'|$, and those sets are
  disjoint for different unavailable vertices (because $G_2$ is a star
  forest), we conclude that the number of unavailable vertices is at
  most
  \begin{displaymath}
    \frac{|B_2'|}{d/2}
    =
    \frac{d |X_2'|}{d/2}
    =
    2 |X_2'|
    \leq
    \frac{|X_1|}{2} \,,
  \end{displaymath}
  and hence the number of available vertices in $X_2 \setminus X_2'$
  is at least $|X_1|/4$.
  
  We now expose the images of the $d$ neighbors of $x_1$.  This is a
  uniformly random $d$-element subset of $B_2 \setminus B_2'$, where
  \begin{displaymath}
    (1-o(1))n = n - d |X_1'| \leq |B_2 \setminus B_2'| \leq n \,.
  \end{displaymath}
  For each available vertex $x_2$, its (deterministically known)
  neighborhood in $B_2 \setminus B_2'$ has size at least $d/2$, and
  there are at least $|X_1|/4 = n^\alpha / 4$ such neighborhoods,
  all disjoint, coming from different available vertices.  
  We are therefore in the setting of Lemma
  \ref{lem:poisson_general} (with $(1-o(1))n$ instead of $n$), and so we conclude that with probability
  $1 - e^{-\Omega(n^{\alpha/4})}$, there is an available vertex $x_2$
  such that $|f_B(N_{x_1}) \cap N_{x_2}| \ge \frac{\alpha \log n}{9
    \log \left(\frac{n \log n}{d^2} \right)}$. Furthermore, we only
  need to expose the outcome of $f_B$ on $N_{x_1}$. We can continue
  the process for at least $\frac{|X_1|}{4}$ times, with probability
  at least $1 - \frac{|X_1|}{4} \cdot e^{-\Omega(n^{\alpha/4})} = 1 -
  e^{-\Omega(n^{\alpha/4})}$. This proves the lemma.
\end{proof}

Our next proposition bounds the self-similarity of a graph in terms of
its median degree.  To prove the proposition, we will find many
star-forests in our graph, and apply Lemma \ref{lem:general_mapping}
several times.

\begin{PROP}
  \label{prop:key} 
  Let $\alpha \leq \frac{1}{25}$ be a fixed positive real number.
  Then for every sufficiently large $n$ and $d$ satisfying $6
  n^{\frac{1}{2} - \frac{\alpha}{16}} \le d \le \sqrt{\alpha n \log
    n}$, every $n$-vertex graph $G$ with at least $\frac{n}{2}$
  vertices of degree at least $d$ has $\ISO(G) > \frac{\alpha n \log
    n}{2592 \log \left( \frac{n \log n}{d^2} \right)}$.
\end{PROP}
\begin{proof}
  Take a uniformly random partition $A_1 \cup A_2 \cup B_1 \cup B_2$
  of the vertex set, where $|A_1| = |A_2| = |B_1| = |B_2| =
  \frac{n}{4}$.  For $i=1,2$, let $G_i$ be the bipartite graph formed
  by the edges between $A_i$ and $B_i$. Since $d > n^{1/3}$, by the
  concentration of the hypergeometric distribution (see, e.g., 
  Theorem 2.10 of \cite{JaLuRu}) and a union bound,
  one can see that a.a.s.\ each $A_i$ contains at least $\frac{n}{9}$
  vertices that have at least $\frac{d}{5}$ neighbors in $B_i$ in the
  graph $G_i$.  Condition on this event.

  Let $d' = \frac{d}{10}$ and $n' = \frac{n}{4}$, and note that since
  $\alpha \leq \frac{1}{25}$, $\frac{2n^\alpha}{3} < (n')^\alpha <
  n^\alpha$.  Let $k_1$ be the largest integer for which we can find a
  collection of $(d', (n')^\alpha)$-star-forests $S_{1,j} = \{(v, N_v)
  : v \in X_{1,j}\}$ in $G_1$, where the sets $X_{1,j}$ are disjoint
  subsets of $A_1$, for $1 \leq j \leq k_1$.  We claim that $k_1 \geq
  \frac{n^{1-\alpha}}{18}$.  Indeed, if not, then there exist over
  $\frac{n}{9} - k_1 (n')^{\alpha} \ge \frac{n}{18}$ vertices in $A_1$
  that are not covered by the sets of the form $X_{1,j}$, and have
  degree at least $\frac{d}{5}$ in the set $B_1$. Let $A_1'$ be the
  set of these vertices. By our maximality assumption, we know that
  the graph $G_1[A_1' \cup B_1]$ does not contain a $(d',
  (n')^{\alpha})$-star-forest.  Let $S = \{(v, N_v) : v \in X\}$ be a
  $(d', h)$-star-forest in $G_1[A_1' \cup B_1]$, where $X\subset A_1'$
  and $h$ is as large as possible. By our assumption, we know that $h
  < (n')^{\alpha}$.  Then all the vertices in $A_1' \setminus X$ have
  degree at least $\frac{d}{10}$ in the set $N = \bigcup_{v \in X}
  N_v$. Note that $|N| = d' h < \frac{d}{10} \cdot n^{\alpha}$ and
  $|A_1' \setminus X| \ge \frac{n}{18} - h > \frac{n}{19}$.  In this
  case, Corollary \ref{cor:volume} applied to $G[(A_1' \setminus X)
  \cup N]$ already gives
  \begin{displaymath}
    \ISO(G)
    \geq
    \ISO(G[(A_1' \setminus X) \cup N])
    \geq
    \frac{\left((d/10) \cdot |A_1' \setminus X|\right)^2}{5|N| \cdot |A_1' \setminus X|}
    = \frac{d^2|A_1' \setminus X|}{500|N|}
    %\ge \frac{d^2 n}{400d n^{\alpha}} 
    > \frac{dn^{1-\alpha}}{950} 
    > \frac{n^{4/3}}{950} \,, 
  \end{displaymath}
  which for large $n$ is already far more than enough.  Therefore, we
  may assume that $k_1 \ge \frac{n^{1-\alpha}}{18}$.  Similarly, there
  is a collection of $\frac{n^{1-\alpha}}{18}$ many $(d',
  (n')^{\alpha})$-star-forests $S_{2,j} = \{(v, N_v) : v \in
  X_{2,j}\}$ in $G_2$, where $X_{2,j}$ are disjoint subsets of $A_2$.

  Let $f_B$ be a bijection from $B_1$ to $B_2$ chosen uniformly at random.
  Our initial conditions
  on $n$ and $d$ imply that $n'$ and $d'$ satisfy the requirements of
  Lemma \ref{lem:general_mapping}, so for each fixed $j$, with
  probability at least $1 - e^{-\Omega(n^{\alpha/4})}$, $f_B$ can be
  extended to a bijection between $B_1 \cup X_{1,j}$ and $B_2 \cup
  X_{2,j}$ such that $f_B(G_1[B_1 \cup X_{1,j}])$ and $G_2[B_2 \cup
  X_{2,j}]$ overlap in at least
  \begin{displaymath}
    |X_{1,j}| \cdot 
    \frac{\alpha \log n'}{36 \log \left( \frac{n' \log (n')}{(d')^2} \right)} 
    >
    \frac{2n^\alpha}{3} \cdot \frac{\alpha \log
      \left(\frac{n}{4}\right) }{36\log \left( \frac{25 n \log n}{d^2} \right)} 
    \geq
    \frac{n^\alpha \cdot \alpha \log
      n}{144 \log \left( \frac{n \log n}{d^2} \right)} 
  \end{displaymath}
  edges, where we used $\frac{n \log n}{d^2} \geq \frac{1}{\alpha} \geq
  25$.

  Since the sets $X_{1,j}$ are disjoint for distinct $j$, and
  $X_{2,j}$ are also disjoint for distinct $j$, a union bound shows
  that we can independently extend the bijection $f_B$ by each
  $X_{1,j} \rightarrow X_{2,j}$ to construct a map $f : A \rightarrow
  B$ which establishes
  \begin{displaymath}
    \ISO(G_1, G_2)
    >
    \frac{n^{1-\alpha}}{18} \cdot
    \frac{n^\alpha \cdot \alpha \log
      n}{144\log \left( \frac{n \log n}{d^2} \right)} 
    = \frac{\alpha n \log n}{2592 \log \left( \frac{n \log n}{d^2} \right)} \,,
  \end{displaymath}
  completing the proof.
\end{proof}

We are now ready to prove Theorem \ref{thm:main}, and establish the
correct order of magnitude of the function $\ISO(m)$.

\begin{proof}[Proof of Theorem \ref{thm:main}]
  Consider the random graph $\Gnp$ with $p = \sqrt{\frac{\log n}{n}}$.
  For $m = \frac{1}{2} n^{3/2} \sqrt{\log n}$, we a.a.s.~have $e(\Gnp)
  = (1+o(1))m$, and by Theorem \ref{thm:randomgraph}, $\ISO(\Gnp) =
  \Theta(n \log n) = \Theta((m\log m)^{2/3})$. Since the function
  $\ISO$ is monotone, this shows that $\ISO(m) \le O((m\log
  m)^{2/3})$, and establishes the upper bound. In the remainder of the
  proof, we focus on proving the lower bound.

  Let $G$ be the given graph with $n$ vertices and $m$ edges.  Without
  loss of generality, we may assume that $G$ contains no isolated
  vertices.  Let $n_0 = n$, $m_0 = m$, $G_0 = G$, and let $V_0$ be the
  vertex set of $G_0$.  Let $n_0 = 2^{a_0} \frac{m_0^{2/3}}{(\log
    m_0)^{1/3}}$ for some real $a_0$.  Let $t=1$ in the beginning and consider the
  following iterative process.  At each step $t$, we will either find
  two large isomorphic edge-disjoint subgraphs, or will find an
  induced subgraph $G_{t}$ on the vertex set $V_t$ such that for $n_t
  = |V_t|$, $m_t = |E(G_{t})|$, and $a_t$ satisfying $n_t = 2^{a_t}
  \frac{m_t^{2/3}}{(\log m_t)^{1/3}}$, we have the following
  properties:
  \begin{enumerate}[(i)]
    \setlength{\itemsep}{1pt} \setlength{\parskip}{0pt}
    \setlength{\parsep}{0pt}
  \item $G_t$ has no isolated vertex,
  \item $m_0 \ge m_t \ge \left(1 - \sum_{i=0}^{t-1} 2^{-a_i} \right)
    m_0 > \frac{m}{3}$, and
  \item $a_t \le a_{t-1} - \frac{1}{3}$ for $t \ge 1$.
  \end{enumerate}
  Note that the properties indeed hold for $t=0$.  Suppose that we are
  given parameters as above for some $t \ge 0$.  If $n_t \ge (m_t \log m_t)^{2/3}$, then by Proposition \ref{prop:starforest}, we have
  $\ISO(G) \ge \frac{(m_t \log m_t)^{2/3} -2}{4} = \Omega( (m \log m)^{2/3} )$.  On the other hand, if $n_t \le \frac{8
    m_t^{2/3}}{(\log m_t)^{1/3}}$, then by Corollary
  \ref{cor:general-volume}, we have
  \begin{displaymath}
    \ISO(G) 
    \ge 
    \frac{m_t^2}{5 n_t^2} 
    %= \frac{1}{320}(m_t\log m_t)^{2/3} 
    \ge \Omega((m\log m)^{2/3}) \,.
  \end{displaymath}
  Therefore, we may assume that
  \begin{align} 
    \label{eq:range} 
    \frac{8 m_t^{2/3}}{(\log m_t)^{1/3}}
    <
    n_t 
    <
    (m_t \log m_t)^{2/3} \,.
  \end{align}
  from which it follows that $3 < a_t < \log_2 \log m_t$.  Define
  \begin{displaymath}
    d_t 
    = 
    \frac{m_t}{2^{a_t} \cdot n_t}
    =
    2^{-2a_t} (m_t \log m_t)^{1/3} \,,
  \end{displaymath}
  and let $V_{t}'$ be the subset of vertices which have degree at
  least $d_t$ in the graph $G_t$. Using the upper bound of
  \eqref{eq:range} together with $a_t < \log_2 \log m_t$, one can see that
  \begin{displaymath}
    d_t 
    >
    \frac{n_t^{3/2} / \log m_t}{2^{a_t} \cdot n_t}
    >
    \frac{n_t^{1/2}}{(\log m_t)^2}
    >
    6 n_t^{\frac{1}{2} - \frac{\alpha}{16}}
  \end{displaymath}
  for $\alpha = \frac{1}{25}$. The lower
  bound of \eqref{eq:range} gives $m_t< (\frac{n_t}{8})^{3/2}(\log
  m_t)^{1/2}$, so using $a_t>3$, we find that
  \begin{displaymath}
    d_t
    \leq
    \frac{(n_t/8)^{3/2} \sqrt{\log m_t}}{2^{a_t} \cdot n_t}
    <
    \frac{1}{147} \sqrt{n_t \log n_t}
    <
    \sqrt{\alpha n_t \log n_t} \,.
  \end{displaymath}
  Consequently, if $|V_{t}'| \ge \frac{|V_t|}{2}$, then by Proposition
  \ref{prop:key} we have 
  \begin{displaymath}
    \ISO(G) 
    > 
    \frac{\alpha n_t \log n_t}{2592 \log \left( \frac{n_t \log n_t}{d_t^2} \right)} 
    =
    \frac{n_t \log n_t}{64800 \log \left( \frac{n_t \log n_t}{d_t^2} \right)} 
\,.
  \end{displaymath}
  Since $\frac{n_t}{d_t^2} = \frac{2^{5a_t}}{\log m_t}$ and $\log m_t
  > \log n_t = a_t \log 2 + \frac{2}{3} \log m_t - \frac{1}{3} \log
  \log m_t > \frac{1}{2}\log m_t$, we have
%\[ \frac{n_t}{d_t^2} = \frac{2^{5a_t}}{\log m_t} \quad \textrm{and} \quad
% \log m_t \ge \log n_t = a_0 \log 2 + \frac{1}{3} \log m_t - \frac{1}{3} \log \log m_t \ge \frac{1}{4}\log m_t, \]
  \begin{displaymath}
    \ISO(G)
    > \frac{n_t \log n_t}{64800 \log \left( \frac{n_t}{d^2}\log n_t \right)}
    > \frac{n_t \log m_t}{(648000 \log 2 )a_t}
    = \frac{2^{a_t}}{(648000 \log 2 ) a_t} \cdot m_t^{2/3}(\log m_t)^{2/3}.
  \end{displaymath}
  Since $a_t > 3$, we have $\frac{2^{a_t}}{a_t} > 2$, and thus
  \begin{displaymath}
    \ISO(G) 
    > 
    \frac{1}{324000\log 2} \cdot m_t^{2/3}(\log m_t)^{2/3} 
    = \Omega((m\log m)^{2/3}) \,.
  \end{displaymath}
  Otherwise, we have $|V_{t}'| < \frac{|V_t|}{2}$. Let $V_{t+1}$ be
  the set of non-isolated vertices in the induced subgraph
  $G[V_t']$. Let $n_{t+1} = |V_{t+1}|$ and let $m_{t+1}$ be the number
  of edges in the induced subgraph $G_{t+1} = G[V_{t+1}]$.  Define
  $a_{t+1}$ so that $n_{t+1} = 2^{a_{t+1}} \frac{m_{t+1}^{2/3}}{(\log
    m_{t+1})^{1/3}}$.  Note that since we only removed vertices whose
  degree in $G_t$ was less than $d_t$, our new number of edges is
  $m_{t+1} > m_t - n_t d_t = (1 - 2^{-a_t}) m_t$, and in particular is
  well above $m_t / 2$ because $a_t > 3$.  Property (i) follows from
  the definition. For Property (ii), note that
  \begin{displaymath}
    m_{t+1} 
    > (1 - 2^{-a_t}) m_t
    \ge (1 - 2^{-a_t}) \left(1 - \sum_{i=0}^{t-1} 2^{-a_i} \right) m_0
    > \left(1 - \sum_{i=0}^{t} 2^{-a_i} \right) m_0 \,,
  \end{displaymath}
  and moreover, since $a_i > 3$ and $a_{i+1} \le a_i - \frac{1}{3}$
  for all $i$, we have
  \begin{displaymath}
    \left(1 - \sum_{i=0}^{t} 2^{-a_i} \right) m 
    >
    \left(1 - \sum_{i=0}^\infty 2^{-3 - \frac{i}{3}} \right) m
    =
    \left(1 - \frac{1}{8} \cdot \frac{1}{1 - 2^{-1/3}} \right) m 
    > 
    \frac{m}{3} \,.
  \end{displaymath}
  Finally, since $m_t/2 < m_{t+1} \leq m_t$ we have
  \begin{displaymath}
    n_{t+1} 
    < \frac{n_t}{2} 
    = 2^{a_t - 1}\frac{m_{t}^{2/3}}{(\log m_{t})^{1/3}}
    < 2^{a_t - 1}\frac{(2m_{t+1})^{2/3}}{(\log m_{t+1})^{1/3}}
    = 2^{a_t - \frac{1}{3}} \frac{m_{t+1}^{2/3}}{(\log m_{t+1})^{1/3}},
  \end{displaymath}
  from which Property (iii) follows.  Note that by Property (iii), at
  some time $s$ we will reach $a_s \leq 3$, and will be done by
  Corollary \ref{cor:general-volume}, in the middle of the process at
  time $s$.
\end{proof}

% [REMARK: Note that the cutoff for using star-forest argument,
% Proposition \ref{prop:starforest}, can be set at $n =
% m^{2/3+\varepsilon}$ for some small $\varepsilon$ instead of the
% current cutoff $n=m^{2/3}\log m$. This shows that we might be able
% to do some inductive argument for $k$-uniform hypergraphs by
% uniformly splitting the graph into $k$-parts, and finding embedding
% one by one, to even gain polynomial factor.]

\section{Concluding remarks}

In this paper, we proved that $\ISO(m) = \Theta( (m\log m)^{2/3}
)$. The upper bound followed by considering the random graph $\Gnp$
with $p = \sqrt{\frac{\log n}{n}}$. For this range of $p$, we have $m
= \Theta( n^{3/2}(\log n)^{1/2} )$, or equivalently $n =
\Theta\big(\frac{m^{2/3}}{(\log m)^{1/3}}\big)$.  By carefully
studying the proof of Theorem 1.2, one can notice that every graph $G$
with $\ISO(G) \le O((m \log m)^{2/3})$ has to be somewhat similar to
the above random graph.  Indeed, by choosing different parameters in
the proof, one can see that for every $\varepsilon > 0$, such graphs
$G$ must contain a subgraph on $n' = \Theta\big(\frac{m^{2/3}}{(\log
  m)^{1/3}}\big)$ vertices with at least $(1-\varepsilon)m$ edges,
where the degree of at least $(1-\varepsilon)n'$ vertices is
$\Omega(d)$, for $d$ being the average degree of the subgraph (thus $d
= \Theta((m\log m)^{1/3})$). Moreover, the edges of this subgraph are
well-distributed, in the sense that there does not exist a pair of
disjoint vertex subsets $X, Y$ satisfying $e(X,Y) \gg d\sqrt{|X| |Y|}$
(since in this case we can directly apply Corollary
\ref{cor:general-volume}).

For a positive integer $s \ge 2$, let $\ISO_{s}(G)$ be the maximum $t$
for which $G$ contains an $s$-divisible subgraph with $t$ edges, and
let $\ISO_{r,s}(m)$ be the minimum of $\ISO_{s}(G)$ over all
$r$-uniform hypergraphs with $m$ edges (thus we have $\ISO_r(m) =
\ISO_{r,2}(m)$). By slightly adjusting our proof of the bound $\ISO(m)
= \Theta((m\log m)^{2/3})$, we can also prove for fixed constant $s$
that $\ISO_{2,s}(m) = \Theta( m^{\frac{s}{2s-1}} (\log
m)^{\frac{2s-2}{2s-1}} )$. The upper bound follows by considering the
random graph $\Gnp$ with $p = \left(\frac{\log
    n}{n}\right)^{1/s}$. For the lower bound, if $n \le
\frac{m^{s/(2s-1)}}{(\log m)^{1/(2s-1)}}$, then we can use an argument
similar to that of Corollary \ref{cor:general-volume}, and if $n \ge
m^{\frac{s}{2s-1}} (\log m)^{\frac{2s-2}{2s-1}} $, then we can use an
argument similar to that of Proposition \ref{prop:starforest}. In the
remaining range of parameters, we can proceed as in Section
\ref{sec:general}. The value $\frac{\alpha \log n}{8 \log (\frac{n\log
    n}{d^2})}$ in Lemma \ref{lem:poisson_general} will be replaced by
$\Omega\left(\frac{\log n}{\log \big(\frac{n^{s-1}}{d^s}\log n \big)}
\right)$.

  %3-uniform hypergraphs: does the method sufficiently increase the
  %power of the logarithm to match random constructions? 
  % [REMARK: we might not want to mention this if we are planning to
  % work on it]

\section*{Acknowledgment}

We would like to thank the referee for careful reading of this paper, and
for useful feedback.

\end{document}